\documentclass{symmetry12}
\usepackage{stmaryrd}

\usepackage[paperwidth=176mm,paperheight=250mm]{geometry} 
\usepackage{hyperref}

\def\Jet{J^\infty}
\def\jet{j^\infty}
\def\F{\mathcal{F}}

\def\dif{\mathrm{d}}
\newcommand\nn{\nonumber}
\newcommand\isempty[3]{{\ifx&#1&#2\else#3\fi}}
\newcommand\term[1]{\emph{#1}}
\newcommand\dvol[1][M]{\mathrm{dVol}(#1)}
\newcommand{\hJet}[1]{\ol{\Jet_{#1}}}
\newcommand{\evder}[2][]{\partial_{#2}^{\ifx&#1&\else(#1)\fi}}
\newcommand{\ol}[1]{\overline{#1}}
\newcommand{\wh}[1]{\widehat{#1}}
\newcommand{\hdif}{\overline{\dif}}
\newcommand{\difv}[2][]{\frac{\delta #1}{\delta #2}} 
\newcommand{\ldifv}[2][]{\frac{\overleftarrow{\delta #1}}{\delta #2}} 
\newcommand{\rdifv}[2][]{\frac{\overrightarrow{\delta #1}}{\delta #2}} 
\newcommand{\lhdifv}[2][]{\overleftarrow{\delta #1}/\delta #2} 
\newcommand{\rhdifv}[2][]{\overrightarrow{\delta #1}/\delta #2} 

\newcommand{\difp}[2][]{\frac{\partial #1}{\partial #2}}
\newcommand{\dtotal}[2][]{D_{#2}\isempty{#1}{}{(#1)}}
\newcommand\schouten[1]{\llbracket #1 \rrbracket}

\begin{document}

\FirstPageHeading{Ringers}

\ShortArticleName{A comparison of definitions for the Schouten bracket on jet spaces} 

\ArticleName{A comparison of definitions for the Schouten bracket on jet spaces}

\Author{Arthemy V. KISELEV and Sietse RINGERS\,\footnote{Corresponding author.}}
\AuthorNameForHeading{A.V. Kiselev and S. Ringers}
\AuthorNameForContents{KISELEV A.V. and RINGERS S.}
\ArticleNameForContents{A comparison of definitions for the Schouten bracket on jet spaces}

\Address{Johann Bernoulli Institute for Mathematics and Computer Science, \\
University of Groningen, P.O. Box 407, 9700 AK Groningen, The Netherlands}
\EmailD{A.V.Kiselev@rug.nl, S.Ringers@rug.nl}

\Abstract{The Schouten bracket (or antibracket) plays a central role in the Poisson formalism and the Batalin-Vilkovisky quantization of gauge systems. There are several (in)equivalent ways to realize this concept on jet spaces. In this paper, we compare the definitions, examining in what ways they agree or disagree and how they relate to the case of usual manifolds.}

\section{Introduction}
The Schouten bracket is a natural generalization of the commutator of vector fields to the fields of multivectors. It was introduced by J.~A.~Schouten~\cite{Ringers:Schouten}, who with A.~Nijenhuis~\cite{Ringers:Nijenhuis} established its main properties. Later it was observed by A.~Lichnerowicz~\cite{Ringers:Lichnerowicz} that the bracket provides a way to check if a bivector $\pi$ on a manifold determines a Poisson bracket via the formula $\schouten{\pi,\pi} = 0$, which was the first intrinsically coordinate-free method to see this and established the use of the bracket in the Poisson formalism.

Historically, the bracket on jet space seems to have been researched in two distinct areas of mathematics and physics, which have been separate for a long time. The first branch is the quantization of gauge systems; here the bracket is known as the \term{antibracket}. It occurs for example in the seminal papers on the BRST and BV formalism, \cite{Ringers:BRST} and~\cite{Ringers:BV} respectively, where it is used to create a nilpotent operator $\schouten{\Omega,\cdot\,}$ providing a resolution of the space of observables. Other occurences of the bracket in this context are~\cite{Ringers:Zinn-Justin}, \cite{Ringers:CF}, \cite{Ringers:HT} and~\cite{Ringers:Witten}, the last of which contains some geometrical interpretation of the bracket.

In the Poisson formalism on jet spaces it was understood in~\cite{Ringers:Gel'fand-Dorfman} that the bracket plays a similar role for recognizing Poisson brackets as on usual manifolds. Concepts such as Hamiltonian operators and the relation of the bracket with the Yang-Baxter equation are developed in~\cite{Ringers:Gel'fand-Dorfman} and~\cite{Ringers:Reyman}; for a review, see~\cite{Ringers:Dorfman}. A version of the bracket that can be restricted to equations was developed in~\cite{Ringers:l*-coverings}. Later, a different, recursive way of defining the bracket, that we will discuss in this paper, was shown in~\cite{Ringers:topical}.

Generalizations to the $\mathbb{Z}_2$-setup and the purely non-commutative setting of the entire theory have been discussed in~\cite{Ringers:Kontsevich}, \cite{Ringers:Olver-Sokolov}, and more recently~\cite{Ringers:ArthemyChristmas}; for a review, see~\cite{Ringers:ArthemyBook}.

The realization that the brackets in these areas of mathematics and physics coincide is not an obvious one. Accordingly, a number of seemingly distinct ways of defining the bracket has been developed, of which the equivalence is not always immediate and sometimes a subtle issue. This paper aims to examine four of those definitions, of which three will turn out to be equivalent when care is taken.

\medskip
The paper is structured as follows. We first recall in section~\ref{Ringers:sec:preliminaries} the notions of horizontal jet spaces and variational multivectors; at this point it will become clear why the definition of the bracket for usual manifold fails in the case of jet spaces. In section~\ref{Ringers:sec:definitions} we first define the Schouten bracket as an odd Poisson bracket; then, after giving some examples of the bracket acting on two multivectors, we show that this definition is equivalent to the recursive one introduced in~\cite{Ringers:topical}. Using the recursive definition we shall prove the Jacobi identity for the bracket, which yields a third definition for the bracket, in terms of graded vector fields and their commutators.

We use the following notation, in most cases matching that from~\cite{Ringers:ArthemyBook}. Let $\pi \colon E \to M$ be a vector bundle of rank $m$ over a smooth real oriented manifold of dimension $n$; in this paper we assume all maps to be smooth. $x^i$ are the coordinates, with indices $i,j,k,\dots$, along the base manifold; $q^\alpha$ are the fiber coordinates with indices $\alpha,\beta,\gamma\dots$. We take the infinite jet space $\pi_\infty \colon \Jet(\pi) \to M$ associated with this bundle; a point from the jet space is then $\theta = (x^i, q^\alpha, q^\alpha_{x^i}, q^\alpha_{x^ix^j},\dots,q^\alpha_\sigma,\dots) \in \Jet(\pi)$, where $\sigma$ is a multi-index. If $s \in \Gamma(\pi)$ is a section of $\pi$ we denote with $\jet(s)$ its infinite jet, which is a section $\jet(s) \in \Gamma(\pi_\infty)$. Its value at $x \in M$ is $\jet_x(s) = (x^i, s^\alpha(x), \difp[s^\alpha]{x^i}(x), \dots,\frac{\partial^{|\sigma|}s^\alpha}{\partial x^\sigma}(x),\dots) \in \Jet(\pi)$.

The evolutionary vector fields, which we will call \emph{vectors}, are then $\evder[q]{\varphi} = \sum_{|\sigma|\geq0}\sum_{\alpha=1}^m D_\sigma(\varphi^\alpha)\difp{q^\alpha_\sigma}$, where $D_\sigma = D_{x^{i_1}}\circ\dots\circ D_{x^{i_k}}$ are (compositions of) the total derivatives. 
Here $\varphi \in \varkappa(\pi) := \Gamma(\pi^*_\infty(\pi)) = \Gamma(\pi)\otimes_{C^\infty(M)}\F(\pi)$, where $\F(\pi)$ is the ring of smooth functions on the jet space. The \emph{covectors} are then $p \in \wh\varkappa(\pi) := \wh{\varkappa(\pi)} := \text{Hom}_{\F(\pi)}(\varkappa(\pi),\ol\Lambda(\pi))$, i.e., linear functions that map vectors to the space of top-level horizontal forms on jet space. We will denote the coupling between covectors and vectors with $\langle p,\varphi \rangle \in \ol\Lambda(\pi)$. The horizontal cohomology, i.e., $\ol\Lambda(\pi)$ modulo the image of the horizontal exterior differential $\hdif$, is denoted by $\ol{H}^n(\pi)$; the equivalence class of $\omega \in \ol\Lambda(\pi)$ is denoted by $\int\omega \in \ol{H}^n(\pi)$. We will assume that the sections are such that integration by parts is allowed and does not result in any boundary terms; for example, the base manifold is compact, or the sections all have compact support, or decay sufficiently fast towards infinity. Lastly, the variational derivative with respect to $q^\alpha$ is $\difv{q^\alpha} = \sum_{|\sigma|\geq0}(-)^\sigma \dtotal\sigma\difp{q^\alpha}$, while the Euler operator is $\delta = \int \dif_{\mathcal{C}}\,\cdot\,$, where $\dif_\mathcal{C}$ is the Cartan differential. 

For a more detailed exposition of these matters, see for example~\cite{Ringers:Olver}, \cite{Ringers:redbook}, \cite{Ringers:ArthemyBook} and~\cite{Ringers:topical}.

\section{Preliminaries}\label{Ringers:sec:preliminaries}
Let $\xi$ be a vector bundle over $\Jet(\pi)$, and suppose $s_1$ and $s_2$ are two sections of this bundle. We say that they are \term{horizontally equivalent} \cite{Ringers:l*-coverings} at a point $\theta \in \Jet(\pi)$ if $\dtotal\sigma(s_1^\alpha) = \dtotal\sigma(s_2^\alpha)$ at $\theta$ for all multi-indices $\sigma$ and fiber-indices $\alpha$. Denote the equivalence class by $[s]_\theta$. The set
\[
	\hJet\pi(\xi) := \{ [s]_\theta \mid s \in \Gamma(\xi), \, \theta \in \Jet(\pi) \}
\]
is called the \term{horizontal jet bundle} of $\xi$. It is clearly a bundle over $\Jet(\pi)$, whose elements above 
 $\theta$ are determined by all the derivatives $s^\alpha_\sigma := D_\sigma(s^\alpha)$ for all multi-indices $\sigma$ and fiber-indices $\alpha$.

Now suppose $\zeta$ is a bundle over $M$. Let us consider the induced vector bundle $\pi^*_\infty(\zeta)$ over $\Jet(\pi)$, and the horizontal jet bundle $\hJet\pi(\pi^*_\infty(\zeta))$.

\begin{proposition}
	As bundles over $\Jet(\pi)$, the horizontal jet bundle $\hJet\pi(\pi_\infty^*(\zeta))$ and $\pi_\infty^*(\Jet(\zeta))$ are equivalent.
\end{proposition}
\begin{proof}
	The pullback bundle $\pi_\infty^*(\Jet(\zeta))$ is as a set equal to
	\begin{align*}
		\pi_\infty^*(\Jet(\zeta)) 
		= \{ (j_x^\infty(q),j_x^\infty(u)) \in \Jet(\pi)\times\Jet(\zeta) \mid x \in M \}.
	\end{align*}
	On the other hand, consider an element $[s]_\theta \in \hJet\pi(\pi_\infty^*(\zeta))$. Thus, $s$ is a section $s \in \Gamma(\pi_\infty^*(\zeta))$. By Borel's theorem, an arbitrary element over $x \in M$ from $\Jet(\pi)$ can be written as $j_x^\infty(q)$ for some $q \in \Gamma(\pi)$. Now define a section $u \in \Gamma(\zeta)$ by $u := \jet(q)^*s = s \circ j^\infty(q)$, i.e., $u(x) = s(j_x^\infty(q))$. Then by the definition of the total derivative, we have
	\begin{align*}
		\difp[u^a]{x^i}(x) = \difp{x^i}\big(s^a\circ \jet(q)\big)(x) = \big(D_{x^i}s^a\big)\big(\jet_x(q)\big),
	\end{align*}
	that is, the partial derivatives of $u$ and the total derivatives of $s$ coincide. This shows that if we define a map by
	\begin{align*}
		[s]_{\jet_x(q)} \mapsto (\jet_x(q), \jet_x(u)) \in \pi_\infty^*(\Jet(\zeta)),
	\end{align*}
	where $u$ is the section associated to $s$ and $q$ as outlined above, then this map is well-defined and smooth. Moreover, since the partial derivatives of a section $u$ at $x$ and the total derivatives of a section $s$ at $\jet_x(q)$ completely define the equivalence classes $\jet_x(u)$ and $[s]_{\jet_x(q)}$ respectively, this map is also a bijection. Lastly, it is clear that as a bundle morphism over $\Jet(\pi)$, it preserves fibers.
\end{proof}
When  $\zeta$ is a bundle over $M$ instead of over $\Jet(\pi)$, and there is no confusion possible, we will abbreviate $\hJet\pi(\pi^*_\infty(\zeta))$ with $\hJet\pi(\zeta)$.

This identification endows the horizontal jet space $\hJet\pi(\zeta)$ with the Cartan connection -- namely the pullback connection on $\pi_\infty^*(\Jet(\zeta))$. Therefore there exist total derivatives $\dtotal{i}$ on the horizontal jet space $\hJet\pi(\zeta)$; in coordinates these are just, denoting the fiber coordinate of $\zeta$ with $u$, the operators
\begin{align*}
	\dtotal{i} = \difp{x^i} + \sum_{\alpha,\sigma}q^\alpha_{\sigma+1_i}\difp{q^\alpha_\sigma} + \sum_{\beta,\tau}u^\beta_{\tau+1_i}\difp{u^\beta_\tau}.
\end{align*}
Thus, instead of the horizontal derivatives $\dtotal[u^\alpha]{\sigma}$ of sections there are now the fiber coordinates $u^\alpha_\sigma$, which have no derivatives along the fiber coordinates: $\difp{q^\alpha_\sigma}u^\beta_\tau = 0$.

Now consider the bundle $\wh\pi: E^*\otimes\Lambda^n(M) \to M$. Then $\pi^*_\infty(\wh\pi) = \pi^*_\infty(E^*) \otimes \ol\Lambda(\pi)$, so that $\wh\varkappa(\pi) = \Gamma(\pi^*_\infty(\wh\pi))$. Thus, the formalism described above is applicable to covectors, so we either take $p$ to be an element from $\wh\varkappa(\pi)$, an actual covector, or $p \in \hJet\pi(\wh\pi)$.

At this point we take the fibers of the bundle $\wh\pi$ and of $\pi^*_\infty(\wh\pi)$, and reverse their parity, $\Pi \colon p \mapsto b$, while we keep the entire underlying jet space intact~\cite{Ringers:Voronov}. The result is the horizontal jet space $\hJet\pi(\Pi\wh\pi)$ with odd fibers over $x$. An element $\theta$ from this space has coordinates
\[
	\theta = (x^i, q^\alpha, q^\alpha_{x^i}, \dots, q^\alpha_\sigma, \dots; b_\alpha, b_{\alpha,x^i}, \dots, b_{\alpha,\sigma}, \dots).
\]
The coupling $\langle p,\varphi\rangle = \sum_\alpha p_\alpha\varphi^\alpha\,\dvol$ extends tautologically to the odd $b$'s, as do the total derivatives: $D_\sigma b_\alpha = b_{\alpha,\sigma}$.

\begin{definition}
	Let $k \in \mathbb{N}\cup \{0\}$. A \term{variational $k$-vector}, or a \term{variational multivector}, is an element of $\ol{H}^n(\pi_\infty^*(\Pi\wh\pi))$, having a density that is $k$-linear in the odd $b$'s or their derivatives (i.e., it is a homogeneous polynomials of degree $k$ in $b_{\alpha,\sigma}$). If $\xi$ is a $k$-vector we will call $k =: \deg(\xi)$ its \emph{degree}. Note that by partial integration, any such $k$-vector $\xi$ can be written as
	\[
		\xi(b) = \int \langle b, A(b,\dots,b)\rangle
	\]
	for some total totally skew-symmetric total differential operator $A$ that takes $k-1$ arguments, takes values in $\varkappa(\pi)$, and is skew-adjoint in each of its arguments (e.g., in the case of a 2-vector, $\int\langle b^1, A(b^2)\rangle = \int\langle b^2,A(b^1)\rangle$).
\end{definition}

Note that this does not imply that \emph{every} density is, or has to be, a homogeneous polynomial of degree $k$; for example, $\int bb_x\,\dvol = \int (bb_x + D_x(bb_xb_{xx}))\,\dvol$.

To \term{evaluate} such a $k$-vector on $k$ covectors $p^1,\dots,p^k$, we proceed as follows: we put each covector in each possible slot, keeping track of the minus sign associated to the permutation, and normalize by the volume of the symmetric group:
\begin{align}\label{Ringers:eq:evaluation}
	\xi(p^1,\dots,p^k) = \frac1{k!}\sum_{s\in S_k}(-)^s\,\xi(p^{s(1)},\dots,p^{s(k)}),
\end{align}
i.e., in the coordinate expression of (the representative of) $\xi$ we replace the $i$-th $b$ that we come across with $p^{s(i)}$ (moving from left to right), and sum over all permutations $s \in S_k$. Thus, under this evaluation $k$-vectors are $k$-linear total differential skew-symmetric functions on $k$ covectors, landing in the horizontal cohomology of the jet space.

\begin{remark}
	Contrary to the case of usual manifolds $M$, where the space of $k$-vectors is isomorphic to $\bigwedge^kTM$, the space of \emph{variational} $k$-vectors does \emph{not} split in such a fashion. As a result, the two formulas
	\[
		\llbracket X, Y\wedge Z\rrbracket = \llbracket X,Y\rrbracket\wedge Z + (-)^{(\deg(X)-1)\deg(Y)}Y\wedge\llbracket X,Z\rrbracket
	\]
	for multivectors $X$, $Y$ and $Z$, and
	\begin{align}\label{Ringers:eq:bi-derivation}
		\llbracket&X_1\wedge\cdots\wedge X_k, Y_1\wedge\cdots\wedge Y_\ell \rrbracket \nn\\
		 &= \sum_{\substack{i\leq i \leq k\\1\leq j\leq \ell}}(-)^{i+j}
		   [X_i,Y_j]\wedge X_1\wedge\cdots\wedge\widehat{X}_i\wedge\cdots\wedge X_k
		   \wedge Y_1\wedge\cdots\wedge \widehat{Y}_j\wedge\cdots\wedge Y_\ell
	\end{align}
	for vector fields $X_i$ and $Y_j$, no longer hold. Both of these formulas provide a way of defining the bracket on usual, smooth manifolds (together with $\schouten{X,f} = X(f)$ for vector fields $X$ and functions $f \in C^\infty(M)$, and $\schouten{X,Y} = [X,Y]$ for vector fields $X$ and $Y$).
	
	To sketch an argument why the space of variational $k$-vectors does not split in this way, take for example a 0-vector $\omega = \int f\,\dvol$ and a 1-vector, which we can write as $\eta = \int\langle b,\varphi\rangle$ for some $\varphi \in \varkappa(\pi)$. How would we define the wedge product $\omega\wedge\eta$ ? Both of the factors contain a volume form and if we just put them together using the wedge product we get 0, so this approach does not work.

	Suppose then we set in this case $\omega\wedge\eta = \int f\langle b,\varphi\rangle$. Now the problem is that $f$ is not uniquely determined by $\omega$ and $\varphi$ is not uniquely determined by $\eta$; both are fixed only up to $\hdif$-exact terms. For example, $\omega = \int f\,\dvol = \int\left(f+\dtotal[g]{i}\right)\dvol$, but $\int f\langle b,\varphi\rangle \ne \int f\langle b,\varphi\rangle + \int\dtotal[g]{i}\langle b,\varphi\rangle$, because the second term is in general not identically zero. 

	Similarly, we have $\eta = \int\langle b,\varphi\rangle = \int\left(\langle b,\varphi\rangle + \hdif(\alpha(b))\right)$, for any linear map $\alpha$ mapping $b$ into $(n-1)$-forms. In the same way as above, this trivial term stops being trivial whenever we multiply it on the left with the density of a 0-vector, say. The difficulty persists for multivectors of any degree $k$ and so there is no reasonable wedge product or splitting.
\end{remark}
\vspace{2cm}

\section{Definitions of the bracket}\label{Ringers:sec:definitions}
\subsection{Odd Poisson bracket}
\begin{definition}\label{Ringers:def:bracket}
	Let $\xi$ and $\eta$ be $k$ and $\ell$-vectors respectively. The \term{variational Schouten bracket} $\schouten{\xi,\eta}$ of $\xi$ and $\eta$ is the $(k+\ell-1)$-vector defined by%
	\footnote{
		To be precise, if $\xi = \int f(b,\dots,b)\,\dvol$ and $\eta = \int g(b,\dots,b)\,\dvol$, where $f$ and $g$ are both homogeneous polynomials in $b_{\alpha,\sigma}$ of degree $k$ and $\ell$ respectively, then the bracket is given by
		\begin{align*}
			\schouten{\xi,\eta}
			&= \int\sum_{\alpha} \left(\rdifv[f]{q^\alpha}\ldifv[g]{b_\alpha}
					 - \rdifv[f]{b_\alpha}\ldifv[g]{q^\alpha}\right)\dvol,
		\end{align*}
		which does not depend on the representatives $f$ and $g$ because $\delta\circ\hdif=0$. This notation, although correct, does not seem to be used in the literature.
	}
	\begin{align}\label{Ringers:eq:schouten}
		\schouten{\xi,\eta}
		&= \int\sum_{\alpha}\left[\rdifv[\xi]{q^\alpha}\ldifv[\eta]{b_\alpha}
				   - \rdifv[\xi]{b_\alpha}\ldifv[\eta]{q^\alpha}\right]
	\end{align}
	in which one easily recognizes a Poisson bracket. Since there are now the anticommuting coordinates $b_\alpha$, we indicate with the arrows above the variational derivatives whether we mean a left or a right derivative (i.e., if we push the variation $\delta b_\alpha$ or $\delta q^\alpha$ through to the left or to the right). The fact that this is a $(k+\ell-1)$-vector comes from the variational derivatives $\difv{b}$ occuring in the expression: if $\eta$ takes $k$ arguments then $\difv[\eta]{b}$ takes $k - 1$ arguments.
\end{definition}

We will use the following two lemmas to calculate examples~\ref{Ringers:example-0-1} through~\ref{Ringers:example-2-2}.

\begin{lemma}\label{Ringers:thm:variational-b-derivative}
	Suppose $\xi = \langle b, A(b,\dots,b) \rangle$ is a $k$-vector. Then
	\begin{align}
		\ldifv[\xi]{b_\alpha} = kA(b,\dots,b)^\alpha.
	\end{align}
\end{lemma}
\begin{proof}
	We calculate
	\begin{align*}
		\delta b_\alpha\ldifv[\xi]{b_\alpha}
		&= \overleftarrow{\delta_b}\xi 
		 = \overleftarrow{\delta_b}\langle b, A(b,\dots,b) \rangle \\
		&= \langle \delta b, A(b,\dots,b) \rangle
			+ \sum_{n=1}^{k-1} \langle b, A(b,\dots,\delta b, \dots,b) \rangle \\
	\intertext{Now $\delta b$ anticommutes with the $b$ left to it, and $A$ is antisymmetric in all of its arguments, so we can switch $\delta b$ with the $b$ on its left, giving two cancelling minus signs. Doing this multiple times, we obtain}
		&= \langle \delta b, A(b,\dots,b) \rangle
			+ \sum_{j=1}^{k-1} \langle b, A(\delta b, b, \dots,b) \rangle \\
	\intertext{Next we first switch $\delta b$ with the $b$ to its left, and then use the fact that $A$ is skew-symmetric in its first argument, again giving two cancelling minus signs:}
		&= \langle \delta b, A(b,\dots,b) \rangle
			+ (k-1) \langle \delta b, A(b, \dots,b) \rangle \\
		&= k\langle \delta b, A(b,\dots,b)\rangle \\
		&= k\,\delta b_\alpha A(b,\dots,b)^\alpha.
	\end{align*}
	The result follows by comparing the coefficients of $\delta b^\alpha$.
\end{proof}

\begin{lemma}
	Let $\xi$ and $\eta$ be $k$ and $\ell$-vectors respectively, so that $\xi = \int\langle b, A(b,\dots,b)\rangle$ and $\eta = \int\langle b, B(b,\dots,b)\rangle$ respectively. Then
	\begin{align}
		\llbracket \xi, \eta \rrbracket = \int\left[(-)^{k(\ell-1)}\ell\evder[q]{B(b,\dots,b)}\xi - (-)^{k-1}k\evder[q]{A(b,\dots,b)}\eta\right]
	\end{align}
\end{lemma}
\begin{proof}
	In the second term of the definition of the Schouten bracket, we first reverse the arrow on the $b$-derivative, giving a sign $(-)^{k-1}$. In the first term, we swap the two factors $(\rhdifv[\xi]{q^\alpha})(\lhdifv[\eta]{b_\alpha})$. For this we have to move the $\ell-1$ $b$'s of $\lhdifv[\eta]{b_\alpha}$ through the $k$ $b$'s of $\rhdifv[\xi]{q^\alpha}$, giving a sign $(-)^{k(\ell-1)}$. Thus
	\begin{align*}
	\schouten{\xi,\eta}
		&= \int \left[ \rdifv[\xi]{q^\alpha}\ldifv[\eta]{b_\alpha} - \rdifv[\xi]{b_\alpha}\ldifv[\eta]{q^\alpha} \right] \\
		&= \int \left[(-)^{k(\ell-1)}\ldifv[\eta]{b_\alpha}\rdifv[\xi]{q^\alpha} - (-)^{k-1}\ldifv[\xi]{b_\alpha}\ldifv[\eta]{q^\alpha} \right] \\
		&= \int \left[(-)^{k(\ell-1)}\ell\dtotal\sigma B(b)^\alpha \difp[\xi]{q^\alpha_\sigma} - (-)^{k-1}k\dtotal\sigma A(b)^\alpha\difp[\eta]{q^\alpha_\sigma} \right] \\
		&= \int\left[(-)^{k(\ell-1)}\ell\evder[q]{B(b)}\xi - (-)^{k-1}k\evder[q]{A(b)}\eta\right].
	\end{align*}
\end{proof}

\begin{example}\label{Ringers:example-0-1}
	Take a one-vector $\varphi \in \varkappa(\pi)$, i.e., $\xi = \langle b, \varphi \rangle$, and let $\mathcal{H} \in \ol{H}^n(\pi)$ be a 0-vector. Then
	\begin{align*}
		\llbracket \mathcal{H}, \varphi \rrbracket = \int\evder[q]{\varphi}\mathcal{H},
	\end{align*}
	i.e., the Schouten bracket calculates the velocity of $\mathcal{H}$ along $\evder[q]\varphi$.
\end{example}

\begin{example}\label{Ringers:example-1-1}
	Suppose $\xi$ and $\eta$ are two one-vectors, i.e., $\xi = \int\langle b, \varphi_1\rangle$ and $\eta = \int\langle b, \varphi_2\rangle$ for some $\varphi_1,\varphi_2 \in \varkappa(\pi)$. Then
	\begin{align*}
		\schouten{\xi,\eta}
		&= \int\left(\evder[q]{\varphi_2}(\xi) - \evder[q]{\varphi_1}(\eta)\right)
		 = \int\left(\evder[q]{\varphi_2}\langle b, \varphi_1\rangle - \evder[q]{\varphi_1}\langle b, \varphi_2\rangle\right) \\
		&= \int\left(\langle b, \evder[q]{\varphi_2}\varphi_1\rangle - \langle b, \evder[q]{\varphi_1}\varphi_2\rangle\right)
	\intertext{which holds because $b$ does not depend on the jet coordinates $q^\alpha$, whence}
		 &= \int\langle b, [\varphi_2,\varphi_1]\rangle
		  = -\int\langle b, [\varphi_1,\varphi_2]\rangle
	\end{align*}
	Thus, in this case the variational Schouten bracket just calculates the ordinary commutator of evolutionary vector fields, up to a minus sign (c.f. equation~\eqref{Ringers:eq:bi-derivation}).
\end{example}

\begin{example}\label{Ringers:example-2-1}
	Suppose the base and fiber are both $\mathbb{R}$, and let $\xi = \int bb_x\,\dif x$ be a (nontrivial) two-vector and $\eta = \int bx^3q_{xx}\,\dif x$ be a one-vector. Then
	\begin{align*}
		\schouten{\xi,\eta}
		 = 0 + \int 2\evder[q]{b_x}(bx^3q_{xx})\,\dif x
		 = 2\int D_x^2(b_x)bx^3\,\dif x
		 = 2\int x^3b_{xxx}b\,\dif x.
	\end{align*}
	We shall return to this example on p.~\pageref{Ringers:example-recursive} (see Example~\ref{Ringers:example-recursive}).
\end{example}

\begin{example}\label{Ringers:example-2-2}
	In this final example, let $\xi = \int bb_x\,\dif x$ again and $\eta = \int q_xbb_x\,\dif x$; then
	\begin{align*}
		\schouten{\xi,\eta}
		&= 0 + 2\int\evder[q]{b_x}(q_xbb_{x})\,\dif x
		 = 2\int D_x(b_x)\cdot bb_x\,\dif x
		 = 2\int bb_xb_{xx}\,\dif x.
	\end{align*}
	Notice the factor 2 standing in front of the answers in the last two examples; it will become important in the next section.
\end{example}

\subsection{A recursive definition}
The second way of defining the bracket, due to I.~Krasil'shchik and A.~Verbovetsky~\cite{Ringers:topical}, is done in terms of the \term{insertion operator}: let $\xi$ be a $k$-vector, and let $p \in \wh\varkappa(\pi)$ or $p \in \hJet\pi(\wh\pi)$ (i.e., $p$ can be either an actual covector or an element from the corresponding horizontal jet space). Denote by $\xi(p)$ or $\iota_p(\xi)$ 
 the $(k-1)$-vector that one obtains by putting $p$ in the rightmost slot of $\xi$:
\begin{align}\label{Ringers:eq:interior-product}
	\xi(p)(b)
	&= \iota_p(\xi)(b)
	 = \xi(\underbrace{b,\dots,b}_{k-1},p) 
	 = \frac1k\sum_{j=1}^k (-)^{k-j}\xi(b,\dots,b,p,b,\dots,b),
\end{align}
where $p$ is in the $j$-th slot. Note that if we were to insert $k-1$ additional elements of $\wh\varkappa(\pi)$ in this expression in this way, we recover formula~\eqref{Ringers:eq:evaluation}.

\begin{lemma}\label{Ringers:thm:interaction-with-vardif}
	If $\xi = \int\langle b, A(b,\dots,b)\rangle$ is a $k$-vector, then
	\begin{align}
		\frac{\overleftarrow{\delta\xi}(p)}{\delta b_\alpha} = \frac{k-1}{k}\ldifv[\xi]{b_\alpha}(p)
		\quad\text{and}\quad
		\frac{\overrightarrow{\delta\xi}(p)}{\delta b_\alpha} = -\frac{k-1}{k}\rdifv[\xi]{b_\alpha}(p).
	\end{align}
\end{lemma}
\begin{proof}
	$\xi(p)$ is a $(k-1)$-vector, so $\overleftarrow{\delta\xi}(p)/\delta b_\alpha = (k-1)A(b,\dots,b,p)^\alpha$ by Lemma~\ref{Ringers:thm:variational-b-derivative}. However, $\xi$ is a $k$-vector, so $(\lhdifv[\xi]{b_\alpha})(p) = (kA(b,\dots,b)^\alpha)(p) = kA(b,\dots,b,p)^\alpha$, from which the first equality of the Lemma follows. The second equality is established by reversing the arrow of the derivative, using the first equality, and restoring the arrow to its original direction again; this results in the extra minus sign in this equality.
\end{proof}

On the other hand, if $p \in \hJet\pi(\wh\pi)$ then $\delta\xi(p)/\delta q^\alpha = (\delta\xi/\delta q^\alpha)(p)$. Indeed, we have $\partial p_{\beta,\tau} / \partial q^\alpha_\sigma = 0$, and if $f$ is one of the densities of a $k$-vector, then the total derivative $D_{x^i}$ and the insertion operator $\iota_p$ commute. For example,
\begin{align*}
	&\iota_p(D_{x^i}b_{\alpha,\sigma}) = \iota_p(b_{\alpha,\sigma+1_i}) = p_{\alpha,\sigma+1_i}
\end{align*}
and
\begin{align*}
	&D_{x^i}(\iota_p(b_{\alpha,\sigma})) = D_{x^i}(p_{\alpha,\sigma}) = p_{\alpha,\sigma+1_i}.
\end{align*}
Thus, from the formula $\difv{q^\alpha} = \sum_{|\sigma|>0}(-)^\sigma D_\sigma\difp{q^\alpha_\sigma}$ for the variational derivative it follows that $\delta\xi(p)/\delta q^\alpha = (\delta\xi/\delta q^\alpha)(p)$.

\begin{theorem}\label{Ringers:thm:recursive}
	Let $\xi$ and $\eta$ be $k$ and $\ell$-vectors, respectively, and $p \in \hJet\pi(\wh{\pi})$. Then
	\begin{align}\label{Ringers:eq:recursive}
		\llbracket \xi, \eta \rrbracket(p)
		&= \frac{\ell}{k+\ell-1}\llbracket \xi, \eta(p) \rrbracket + (-)^{\ell-1}\frac{k}{k+\ell-1}\llbracket \xi(p), \eta \rrbracket.
	\end{align}
\end{theorem}
\begin{proof}
	We relate the two sides of the equation by letting $p$ range over the slots as in equation~\eqref{Ringers:eq:interior-product}. In this calculation we will for brevity omit the fiber indices $\alpha$.
	
	Consider the first term of the left hand side, $\left(\rdifv[\xi]{q}\ldifv[\eta]{b}\right)(p)$. If we were to take the sum as in equation~\eqref{Ringers:eq:interior-product}, we would obtain an expression containing $k+\ell-1$ slots; in some cases $p$ is in one of the $\ell-1$ slots of $\lhdifv[\eta]{b}$ and in the other cases it is in one of the $k$ slots of $\rhdifv[\xi]{q}$. All of these terms carry the normalizing factor $1/(k+\ell-1)$. Now we notice the following:
	
	\begin{itemize}
		\item Each term in which $p$ is in a slot coming from $\lhdifv[\eta]{b}$ has a matching term in the expansion of $\rdifv[\xi]{q}\,\iota_p\!\left(\ldifv[\eta]{b}\right)$ according to~\eqref{Ringers:eq:interior-product}, except that there each term would carry a factor $1/(\ell-1)$, because now $p$ only has access to the $\ell-1$ slots of $\lhdifv[\eta]{b}$.
		\item Similarly, each term of the left hand side of~\eqref{Ringers:eq:recursive} in which $p$ is in one of the slots of $\rhdifv[\xi]{q}$ has a matching term in the expansion of $\iota_p\!\left(\rdifv[\xi]{q}\right)\ldifv[\eta]{b}$, but there they carry a factor $1/k$.
		\item Moreover, in that case they also carry the sign $(-)^{\ell-1}$, which comes from the fact that here $p$ had to pass over the $\ell-1$ slots of $\lhdifv[\eta]{b}$.
	\end{itemize}
	Gathering these remarks, we find
	\begin{align*}
		\left(\rdifv[\xi]{q}\,\ldifv[\eta]{b}\right)(p)
		&= \frac{\ell-1}{k+\ell-1}\rdifv[\xi]{q}\,\iota_p\!\left(\ldifv[\eta]{b}\right)
		   + (-)^{\ell-1}\frac{k}{k+\ell-1}\iota_p\!\left(\rdifv[\xi]{q}\right)\ldifv[\eta]{b} \\
		&= \frac{\ell}{k+\ell-1}\rdifv[\xi]{q}\frac{\overleftarrow{\delta\eta}(p)}{\delta b}
		   + (-)^{\ell-1}\frac{k}{k+\ell-1}\frac{\overrightarrow{\delta\xi}(p)}{\delta q}\ldifv[\eta]{b},
	\end{align*}
	where we have used the first equation of Lemma~\ref{Ringers:thm:interaction-with-vardif} in the first term.

	Now we consider the second term of the left hand side of~\eqref{Ringers:eq:recursive}, and use a similar reasoning:
	\begin{align*}
		\left(\rdifv[\xi]{b}\,\ldifv[\eta]{q}\right)(p)
		={}& \frac{\ell}{k+\ell-1}\rdifv[\xi]{b}\,\iota_p\!\left(\ldifv[\eta]{q}\right)
		   + (-)^{\ell}\frac{k-1}{k+\ell-1}\iota_p\!\left(\rdifv[\xi]{b}\right)\ldifv[\eta]{q} \\
		={}& \frac{\ell}{k+\ell-1}\rdifv[\xi]{b}\frac{\overleftarrow{\delta\eta}(p)}{\delta q}
		   +(-)^{\ell+1} \frac{k}{k+\ell-1}\frac{\overrightarrow{\delta\xi}(p)}{\delta b}\ldifv[\eta]{q}.
	\end{align*}
	where now the second equation of Lemma~\ref{Ringers:thm:interaction-with-vardif} has been used. Subtracting the results of these two calculations, we obtain exactly the right hand side of equation~\eqref{Ringers:eq:recursive}.
\end{proof}

Thus, by recursively reducing the degrees 
of the arguments of the bracket, formula~\eqref{Ringers:eq:recursive} expresses the value of the bracket of a $k$-vector and an $\ell$-vector on $k+\ell-1$ covectors. We can interpret it as a second definition of the Schouten bracket, provided that we also set
\begin{align*}
	\llbracket \mathcal{H}, \varphi \rrbracket 
	= \int\evder[q]{\varphi}\mathcal{H}
	= \int\langle\delta\mathcal{H},\varphi\rangle
\end{align*}
for $1$-vectors $\varphi$ and 0-vectors $\mathcal{H} \in \ol{H}^n(\pi)$. Theorem~\ref{Ringers:thm:recursive} then says that this definition is equivalent to Definition~\ref{Ringers:def:bracket}. However, let us notice the following:

\begin{remark}
	There are numerical factors in front of the two terms of the right hand side; these are absent in~\cite{Ringers:topical}. For example, the bracket of a $2$-vector $\xi$ and a 0-vector $\mathcal{H}$ is $\schouten{\mathcal{H},\xi}(p) = 2\xi(\delta\mathcal{H},p)$ according to both Definition~\ref{Ringers:def:bracket} and Theorem~\ref{Ringers:thm:recursive}; note the factor 2.
\end{remark}

\begin{remark}
	Secondly, it is important that the $p$ that is inserted in~\eqref{Ringers:eq:recursive} is \emph{not} an actual covector, but that $p \in \hJet\pi(\wh\pi)$. Otherwise, unwanted terms like $\evder[q]{\varphi}(p)$ occur in the final steps, and equivalence with Definition~\ref{Ringers:def:bracket} is spoiled.
	Thus one takes two multivectors, inserts elements from the horizontal jet space according to the formula, and only plugs in the (derivatives of) actual covectors at the end of the day. This remark is again absent from~\cite{Ringers:topical}.
\end{remark}

\begin{example}\label{Ringers:example-recursive}
	Let us re-calculate Example~\ref{Ringers:example-2-1} using this formula. So, let $\xi = \int bb_x\,\dif x$ and $\eta = \int bx^3q_{xx}\,\dif x$, and let $p^1, p^2 \in \hJet\pi(\wh\pi)$. Then
	\begin{align*}
		\schouten{\xi,\eta}&(p^1,p^2)
		 = \schouten{\xi,\eta}(p^2)(p^1)
		 = \frac12\schouten{\xi, \eta(p^2)}(p^1)
		   + \frac22\schouten{\xi(p^2), \eta}(p^1) \\
		&= -2\cdot\frac12\cdot\schouten{\xi(p^1), \eta(p^2)}
		   + 1\cdot\schouten{\xi(p^2), \eta(p^1)}
		   + 1\cdot\schouten{\xi(p^1,p^2), \eta} \\
		&= \int\Big[(-)^2\,\evder[q]{p^1_x}(p^2x^3q_{xx})
		   - \evder[q]{p^2_x}(p^1x^3q_{xx})
		   + {\textstyle\frac12}\evder[q]{x^3q_{xx}}(p^1p^2_x - p^2p^1_x)\Big]\,\dif x \\
		&= \int\left[x^3p^1_{xxx}p^2 - (p^1\rightleftarrows p^2)\right]\dif x.
	\end{align*}
	(Keeping track of the coefficients and signs is a good exercise.) This is precisely what one gets after evaluating the result of Example~\ref{Ringers:example-2-1} on $p^1$ and $p^2$.
\end{example}

Theorem~\ref{Ringers:thm:recursive} allows us to reduce the Jacobi identity for the Schouten bracket to that of the commutator of one-vectors.

\begin{proposition}\label{Ringers:thm:schouten-jacobi}
	Let $r$, $s$ and $t$ be the degrees of the variational multivectors $\xi$, $\eta$ and $\zeta$, respectively. The Schouten bracket satisfies the graded Jacobi identity\textup{:}
	\begin{align}\label{Ringers:eq:schouten-jacobi} 
		(-)&^{(r-1)(t-1)}\llbracket\xi,\llbracket\eta,\zeta\rrbracket\rrbracket \nn\\
		&+ (-)^{(r-1)(s-1)}\llbracket\eta,\llbracket\zeta,\xi\rrbracket\rrbracket \nn\\
		&+ (-)^{(s-1)(t-1)}\llbracket\zeta,\llbracket\xi,\eta\rrbracket\rrbracket = 0.
	\end{align}
\end{proposition}
\begin{proof}
	We proceed by induction using Theorem~\ref{Ringers:thm:recursive}. When the degrees of the three vectors do not exceed 1, the statement follows from the reductions of the Schouten bracket to known structures, as in Examples~\ref{Ringers:example-0-1} and~\ref{Ringers:example-1-1}. Now let the degrees be arbitrary natural numbers. Denote by $I_1$, $I_2$ and $I_3$ the respective terms of the left hand side of~\eqref{Ringers:eq:schouten-jacobi}. Then for any $p \in \hJet\pi(\wh\pi)$ we have that
	\begin{align*}
		I_1(p)
		&= (-)^{(r-1)(t-1)}\big\llbracket\xi,\llbracket\eta,\zeta\rrbracket\big\rrbracket(p)\\
		&= \frac{(-)^{(r-1)(t-1)}}{r+s+t-2}\left( (s+t-1)\big\llbracket\xi,\llbracket\eta,\zeta\rrbracket(p)\big\rrbracket + r(-)^{s+t-2}\big\llbracket\xi(p),\llbracket\eta,\zeta\rrbracket\big\rrbracket\right) \\
		&= \frac{(-)^{(r-1)(t-1)}}{r+s+t-2}\Big( t\big\llbracket\xi,\llbracket\eta,\zeta(p)\rrbracket\big\rrbracket + s(-)^{t-1}\big\llbracket\xi,\llbracket\eta(p),\zeta\rrbracket\big\rrbracket \\
		&\hspace{2.8cm} +\left.r(-)^{s+t-2}\big\llbracket\xi(p),\llbracket\eta,\zeta\rrbracket\big\rrbracket\right).
	\end{align*}
	Similarly,
	\begin{align*}
		I_2(p) &= \frac{(-)^{(r-1)(s-1)}}{r+s+t-2}\Big( r\big\llbracket\eta,\llbracket\zeta,\xi(p)\rrbracket\big\rrbracket + t(-)^{r-1}\big\llbracket\eta,\llbracket\zeta(p),\xi\rrbracket\big\rrbracket \\
		&\hspace{2.8cm} + s(-)^{r+t-2}\big\llbracket\eta(p),\llbracket\zeta,\xi\rrbracket\big\rrbracket\Big), \\
		I_3(p) &= \frac{(-)^{(s-1)(t-1)}}{r+s+t-2}\left( s\big\llbracket\zeta,\llbracket\xi,\eta(p)\rrbracket\big\rrbracket + r(-)^{s-1}\big\llbracket\zeta,\llbracket\xi(p),\eta\rrbracket\big\rrbracket\right. \\
		&\hspace{2.8cm} +\left.t(-)^{r+s-2}\big\llbracket\zeta(p),\llbracket\xi,\eta\rrbracket\big\rrbracket\right).
	\end{align*}
	For notational convenience, let us set $I_1(p) + I_2(p) + I_3(p) =: I/(r+s+t-2)$. Next we rearrange the terms in $I$:
	\begin{align*}
		I ={}&
		(-)^{r-1}t\Big\{
		(-)^{(r-1)(t-2)}\big\llbracket\xi,\llbracket\eta,\zeta(p)\rrbracket\big\rrbracket
		+ (-)^{(r-1)(s-1)}\big\llbracket\eta,\llbracket\zeta(p),\xi\rrbracket\big\rrbracket \\
		&+ (-)^{(s-1)(t-2)}\big\llbracket\zeta(p),\llbracket\xi,\eta\rrbracket\big\rrbracket
		\Big\}
		+ (-)^{t-1}s\Big\{
		(-)^{(r-1)(t-1)}\big\llbracket\xi,\llbracket\eta(p),\zeta\rrbracket\big\rrbracket \\
		&+ (-)^{(r-1)(s-2)}\big\llbracket\eta(p),\llbracket\zeta,\xi\rrbracket\big\rrbracket 
		+ (-)^{(s-2)(t-1)}\big\llbracket\zeta,\llbracket\xi,\eta(p)\rrbracket\big\rrbracket
		\Big\} \\
		&+ (-)^{s-1}r\Big\{
		(-)^{(r-2)(t-1)}\big\llbracket\xi(p),\llbracket\eta,\zeta\rrbracket\big\rrbracket 
		+ (-)^{(r-2)(s-1)}\big\llbracket\eta,\llbracket\zeta,\xi(p)\rrbracket\big\rrbracket \\ 
		&+ (-)^{(s-1)(t-1)}\big\llbracket\zeta,\llbracket\xi(p),\eta\rrbracket\big\rrbracket
		\Big\},
	\end{align*}
	i.e., we obtain the Jacobi identity for $\xi$, $\eta$ and $\zeta(p)$; for $\xi$, $\eta(p)$ and $\zeta$; and for $\xi(p)$, $\eta$ and $\zeta$ (each times some unimportant factors).
	\iffalse
	\begin{itemize}
	\item The first, fifth and last term of the right hand side is $\frac{(-)^{r-1}t}{r+s+t-2}$ times the Jacobi identity for $\xi$, $\eta$ and $\zeta(p)$,
	\item The second, sixth and seventh terms are $\frac{(-)^{t-1}s}{r+s+t-2}$ times the Jacobi-identity for $\xi$, $\eta(p)$ and $\zeta$,
	\item The third, fourth and eighth terms are $\frac{(-)^{s-1}r}{r+s+t-2}$ times the Jacobi identity for $\xi(p)$, $\eta$ and $\zeta$.
	\end{itemize}\fi
	Thus we see that if we know that the identity holds for $(r-1,s,t)$, $(r,s-1,t)$ and $(r,s,t-1)$, then it holds for $(r,s,t)$.
\end{proof}

\subsection{Graded vector fields}
\begin{proposition}\label{Ringers:thm:graded-vector-fields}
	If $\xi$ and $\eta$ are $k$ and $\ell$-vectors respectively, then their Schouten bracket is equal to
	\begin{align}
		\llbracket \xi, \eta \rrbracket = \int Q^\xi(\eta)
		 = \int(\xi)\overleftarrow{Q}^\eta,
	\end{align}
	where for any $k$-vector $\xi$, the graded evolutionary vector field $Q^\xi$ is defined by
	\begin{align}\label{Ringers:eq:Q}
		Q^\xi := \evder[q]{-\rhdifv[\xi]{b}} + \evder[b]{\rhdifv[\xi]{q}}.
	\end{align}
\end{proposition}
\begin{proof}
	This is readily seen from the equalities
	\begin{align*}
		\int\evder[b]{\rhdifv[\xi]{q}}(\eta)
		&= \int\sum_{\alpha,\sigma}\dtotal\sigma\left(\rdifv[\xi]{q^\alpha}\right)\difp[\eta]{b_{\alpha,\sigma}} 
		= \int\sum_{\alpha,\sigma}\rdifv[\xi]{q^\alpha}(-)^\sigma \dtotal\sigma\difp[\eta]{b_{\alpha,\sigma}} \\
		&= \int\sum_\alpha \rdifv[\xi]{q^\alpha} \ldifv[\eta]{b_\alpha},
	\end{align*}
	which is the first term of the Schouten bracket $\schouten{\xi,\eta}$. The second term of~\eqref{Ringers:eq:Q} is done similarly.
\end{proof}

As a consequence of Proposition~\ref{Ringers:thm:graded-vector-fields}, the Schouten bracket is a derivation: if $\eta$ is a product of $k$ factors, then $\llbracket\xi, \eta\rrbracket = \int Q^\xi(\eta)$ has $k$ terms, where in the $i$-th term, $Q^\xi$ acts  on the $i$-th factor while leaving the others alone. However, while the bracket is a derivation in both of its arguments separately, it is \emph{not} a bi-derivation (i.e., a derivation in both arguments simultaneously), as in equation~\eqref{Ringers:eq:bi-derivation}. To see why this is so, take a multivector $\eta$ and let us suppose for simplicity that it has a density that consists of a single term containing $\ell$ coordinates, which can be either $q$'s or $b$'s: $\eta = \prod_{i=1}^\ell a_i$, for a set of letters $a_i$. Then the $i$-th term of $\schouten{\xi,\eta} = \int Q^\xi(\eta)$ is a sign which is not important for the present purpose, times $a_1\cdots Q^\xi(a_i)\cdots a_\ell$.

Now suppose that $\xi = \prod_{j=1}^k c_j$ for some set of letters $c_j$, and note that $Q^\xi(a_i) = (\xi)\overleftarrow{Q}^{a_i} + \text{trivial terms}$. Let us call the trivial term $\omega$ for the moment. Then we see that 
\begin{align*}
	a_1\cdots Q^\xi(a_i)\cdots a_\ell
	&= a_1\cdots (\xi)\overleftarrow{Q}^{a_i}\cdots a_\ell + a_1\cdots\omega\cdots a_\ell
\end{align*}
Here the first term expands to what it should be in order for the bracket to be a bi-derivation, namely a sum consisting of terms of the form
\[
	a_1\cdots c_1\cdots(c_j)\overleftarrow{Q}^{a_i}\cdots c_k\cdots a_\ell
\]
times possible minus signs. The second term, however, is generally no longer trivial, so that it does not vanish. Therefore the bracket is not in general a bi-derivation.

\begin{theorem}\label{Ringers:thm:commutator}
	The Schouten bracket is related to the graded commutator of graded vector fields as follows\textup{:}
	\begin{align}\label{Ringers:eq:commutator}
		\int Q^{\llbracket\xi, \eta\rrbracket}f = \int[Q^\xi,Q^\eta]f,
	\end{align}
	for any smooth function $f$ on the horizontal jet space $\hJet\pi(\wh\pi)$.
\end{theorem}
\begin{proof}
	From the definition of the graded commutator and equation~\eqref{Ringers:eq:Q} we infer that
	\begin{align*}
		\int [Q^\xi,Q^\eta]f
		&= \int Q^\xi(Q^\eta(f)) - (-)^{(k-1)(\ell-1)}\int Q^\eta(Q^\xi(f)) \\
		&= \big\llbracket\xi,\llbracket\eta,f\rrbracket\big\rrbracket
		   - (-)^{(k-1)(\ell-1)}\big\llbracket\eta,\llbracket\xi,f\rrbracket\big\rrbracket
		 = \big\llbracket\llbracket\xi,\eta\rrbracket,f\big\rrbracket \\
		&= \int Q^{\llbracket\xi,\eta\rrbracket}f.
	\end{align*}
	where we used the Jacobi identity in the third line.
\end{proof}

This provides a third way of defining the Schouten bracket, equivalent to the previous two. Since the only fact that is used in this proof is that the Schouten bracket satisfies the graded Jacobi identity (Proposition~\ref{Ringers:thm:schouten-jacobi}), Theorem~\ref{Ringers:thm:commutator} is actually equivalent to the Jacobi identity for the Schouten bracket. It is also possible to prove Theorem~\ref{Ringers:thm:commutator} directly (see~\cite[p.~84]{Ringers:ArthemyBook}, by inspecting both sides of equation~\eqref{Ringers:eq:commutator}; in that case the Jacobi identity may be proved as a consequence of Theorem~\ref{Ringers:thm:commutator}.

As a bonus, we see that if $P$ is a Poisson bi-vector, i.e., $\schouten{P,P} = 0$, then $Q^P$ is a differential, $(Q^P)^2 = 0$. This gives rise to the Poisson(-Lichnerowicz) cohomology groups $\text{H}_P^k$.

\section{Conclusion}\label{Ringers:sec:conclusion}
The research into the generalization of the Schouten bracket to jet spaces has historically been split in two directions. In the Poisson formalism, it is related to notions such as Poisson cohomology, integrability and the Yang-Baxter equation; while in the quantization of gauge system it is used in the BV-formalism to create a differential $D = \schouten{\Omega,\cdot\,}$, also leading to cohomology groups. Although the definition of the bracket on usual manifolds by the formula that expresses it as a bi-derivation no longer works, there are several other ways of defining the bracket, which are equivalent if care is taken.

We finally recall that these definitions 
of the Schouten bracket also exist and remain coinciding in the $\mathbb{Z}_2$-graded setup $\Jet\big((\pi_0|\pi_1)\big) \to M^{n_0|n_1}$, and in the setup of purely non-commutative manifolds and non-commutative bundles (see~\cite{Ringers:Kontsevich}, \cite{Ringers:Olver-Sokolov} and lastly~\cite{Ringers:ArthemyBook}, which contains details and discussion, and generalizes the topic of this paper to the non-commutative world).

\subsection*{Acknowledgements}
Both authors thank the Organizing committee of this workshop for partial support and a warm atmosphere during the meeting. The research of the first author was partially supported by NWO VENI grant 639.031.623 (Utrecht) and JBI RUG project 103511 (Groningen).

\LastPageEnding\end{document}